\documentclass[12pt]{amsart}
\usepackage{geometry}   %????????
\usepackage[colorlinks,citecolor = red, linkcolor=blue,hyperindex]{hyperref}
\usepackage{euscript,eufrak,verbatim, mathrsfs}
\usepackage[psamsfonts]{amssymb}
\usepackage{bbm}
\usepackage{graphicx}
 \usepackage{float}
\usepackage{float, tikz}

 \usepackage[all, cmtip]{xy}

\usepackage{upref, xcolor, dsfont}
\usepackage{amsfonts,amsmath,amstext,amsbsy, amsopn,amsthm}
\usepackage{enumerate}

\usepackage{url}

\usepackage{mathtools}
\usepackage{bookmark}

 \usepackage{euscript}
\usepackage{helvet}         % selects\textbf{\textbf{â¢}} Helvetica as sans-serif font
\usepackage{courier}        % selects Courier as typewriter font
\usepackage{type1cm}        % activate if the above 3 fonts are
\usepackage{multicol}        % used for the two-column index
\usepackage[bottom]{footmisc}% places footnotes at page bottom

\newtheorem{theorem}{Theorem}[section]
\newtheorem*{theorem*}{Theorem B}
\newtheorem{lemma}[theorem]{Lemma}

\newtheorem{proposition}[theorem]{Proposition}
\newtheorem{corollary}[theorem]{Corollary}

\newtheorem*{definition*}{Definition}
\newtheorem*{remark*}{Remark}

\newtheorem*{observation*}{Observation}

\newtheorem*{assumption*}{Assumption}
\newtheorem*{question*}{Question}
\newtheorem*{problem*}{Problem}
\newtheorem{remark}[theorem]{Remark}

\newcommand{\N}{\mathbb{N}}

\newcommand{\E}{\mathbb{E}}

\newcommand{\PP}{\mathbb{P}}

\newcommand{\EE}{\mathcal{E}}
\newcommand{\NN}{\mathcal{N}}

\begin{document}

\title[Random integer sequence]{On a class of random sets of positive integers}

\author%[authorlabel1]
{Yong Han}
\address%[authorlabel1]
{Yong HAN: College of Mathematics and Statistics, Shenzhen University, Shenzhen 518060, Guangdong, China. }
\email{hanyongprobability@gmail.com}

\author%[authorlabel1]
{Yanqi Qiu}
\address%[authorlabel1]
{Yanqi QIU: Institute of Mathematics and Hua Loo-Keng Key Laboratory of Mathematics, AMSS, Chinese Academy of Sciences, Beijing 100190, China.}
\email{yanqi.qiu@amss.ac.cn}

\author{Zipeng Wang}
\address{Zipeng WANG: College of Mathematics and Statistics, Chongqing University, Chongqing,
401331, P.R.China}
\email{zipengwang2012@gmail.com, zipengwang@cqu.edu.cn}

\thanks{Y. Qiu is supported by grants NSFC Y7116335K1,  NSFC 11801547 and NSFC 11688101 of National Natural Science Foundation of China. Z. Wang is supported by NSFC 11601296}

\begin{abstract}
In this note, we study a class of random subsets of positive integers induced by Bernoulli random variables. We obtain sufficient conditions such that the random set is almost surely lacunary, does not have bounded gaps and contains infinitely many arithmetic progressions, respectively.
\end{abstract}


\maketitle

\setcounter{equation}{0}

\section{Introduction}
Let $\mathbb{N} = \{1, 2, \cdots \}$ be the set of positive integers. Suppose that $X=(X_k)_{k\in\mathbb{N}}$ is a sequence of independent Bernoulli random variables with
\[
\PP[X_k=1]=1-\PP[X_k=0]=p_k\in [0,1], \quad \forall k \in \N.
\]
In this note, we studied the random set $\EE_X \subset \N$ defined by
\[
\EE_X: =\{k\in\N: X_k=1\}.
\]
By Borel-Cantelli lemma,   $\sum_{k=0}^{\infty}p_k=\infty$ if and only if 
$|\EE_X|=\infty$ almost surely. Here  by $|\EE_X|$, we mean the cardinality of  the set $\EE_X$. We shall always consider the case when $\EE_X$  is almost surely an infinite subset. 

Motivated by recent constructions of number rigid determinantal point processes
on the unit disc $\mathbb{D}$ with sub-Bergman kernels \cite{QW}, the first part of this notes is devoted to investigation of the  sparse properties of the random set $\EE_X$. We give sufficient conditions on $(p_k)_{k\in\N}$ such that $\EE_X$ is lacunary or does not have bounded gap.

Recall that an increasing sequence $S=\{n_k\}\subset \N$ is a {\it lacunary sequence} if
\[
\liminf_{k}\frac{n_{k+1}}{n_k}> 1.
\]

We call that a function  $f:[1,\infty)\to [1, \infty)$ is {\it admissible} if
\begin{itemize}
\item $f$ is non-decreasing.
\item  For any $\epsilon>0$, we have
\[
\int_1^{\infty}\frac{dx}{xf(x)}=\infty \text{ and } \int_1^{\infty}\frac{dx}{xf(x)^{1+\epsilon}}<\infty.
\]
\end{itemize}

\begin{theorem}\label{P:lacunary}
Let $f: [1, \infty) \rightarrow [1, \infty)$ be an admissible function.  Let $X=(X_k)_{k\in\mathbb{N}}$ be a sequence of independent Bernoulli random variables with
$$\PP[X_k=1]=1-\PP[X_k=0]=\frac{1}{kf(k)}.$$ Let
$\NN_n=| \EE_X\cap (2^n,2^{n+1}]|$,
then almost surely, we have
\begin{align}\label{sup-prod}
\limsup\limits_{n\to\infty}\NN_n\cdot \NN_{n+1}=0.
\end{align}
In particular, almost surely, $\EE_X$ is a lacunary sequence.
\end{theorem}

Given an increasing infinite sequence $S=\{n_k\}\subset \N$, define the {\it gap} of $S$ as
\[
\mathrm{Gap}(S):=\sup_{i\in\N}(n_{k+1}-n_k).
\]
If $\mathrm{Gap}(S)<\infty$, we say that $S$ has bounded gap.

\begin{proposition}\label{P:gap}
Suppose that $X=(X_k)_{k\in\mathbb{N}}$ is a sequence of independent Bernoulli random variables with
$$
\PP[X_k=1]=1-\PP[X_k=0]=p_k\in [0,1].
$$
Assume tha $\sum_{k}p_k=\infty$ and  $\sum_k p_k^2<\infty$. Then almost surely $\EE_X$ does not have bounded gap.
That is, if we write $\EE_X$ as an increasing sequence $\EE_X = \{n_k(X) \}$, then
\[
\lim_{k\rightarrow\infty}(n_{k+1}(X)-n_k(X))=\infty, \quad \mathrm{a.s.}
\]
\end{proposition}

Next, we include a result on the intersection of $\EE_X$ and sequences with bounded gap.
\begin{proposition}\label{P:boundedgap}
Let $(p_k)_{k\in\N}$ be a non-increasing sequence in $[0,1]$  and $\sum_{k}p_k=\infty$.  Let $X=(X_k)_{k\in\mathbb{N}}$ be a sequence of independent Bernoulli random variables with
$$
\PP[X_k=1]=1-\PP[X_k=0]=p_k,
$$
Then for any infinite subset $S\subset\N$ with bounded gap, we have
\[
\PP[S\cap\EE_X\neq\emptyset]=1.
\]
\end{proposition}

\begin{remark}
The condition that $p_k$ is non-increasing in Proposition \ref{P:boundedgap} in general can not be removed. For example, take $p_k=\frac{1}{k}$ when $k$ is odd and $p_k=0$ when $k$ is even, then $\EE_X \cap 2\N =\emptyset$.
\end{remark}

Finding arithmetic progressions of  fixed length (known as Erd\H{o}s and Tur\'{a}n's problem) is one of the most attractive question in number theory. In \cite{Yoshiharu-1996}, Yoshiharu Kohayakawa, Tomasz {\L}uczak, and Vojt\v{e}ch R\"{o}dl proved a random-set analogue of Roth's theorem on 3-term arithmetic progressions. The random counterpart of Erd\H{o}s and Tur\'{a}n's problems attracts many attentions(cf. \cite{Mariah-2008,Gowers-2016, Mathias-2016,Liu-2019}). However, it seems that our random set $\EE_X$ is different from these mentioned ones, and it is natural to consider arithmetic progressions properties of $\EE_X$. We present two simple conditions such that almost surely $\EE_X$ contains or does not contain
infinite number of  arithmetic  progressions of a fixed length, respectively.

\begin{proposition}\label{P:AP01}
Let $X=(X_k)_{k\in\mathbb{N}}$ be a sequence of independent Bernoulli random variables with
$$\PP[X_k=1]=1-\PP[X_k=0]=p_k=\frac{1}{e^{c(\log n)^\epsilon}}$$
for some constant $c$ and $\epsilon\in(0,1)$. Then almost surely $\EE_X$ contains infinite number of  arithmetic  progressions of arbitrary length.
\end{proposition}

\begin{proposition}\label{P:AP02}
Let $X=(X_k)_{k\in\mathbb{N}}$ be a sequence of independent Bernoulli random variables with
$$\PP[X_k=1]=1-\PP[X_k=0]=p_k.$$
Suppose that $p_k=O(k^{-\alpha})$ for $0 < \alpha \le 1/2$ and let $l \ge 2$ be the smallest integer with $l >  \alpha^{-1} \ge 2$.   Then almost surely $\EE_X$ does not contain infinite number of  arithmetic  progressions of length $l + 1$.
\end{proposition}

\section{Lacunary property}

This section is devoted to the proof of Theorem \ref{P:lacunary}. 

\begin{lemma} \label{P:Cmany}
Let $\alpha>0$ and let $f:[1,\infty)\to [1, \infty)$ be an admissible function.  Suppose that $X=(X_k)_{k\in\mathbb{N}}$ are independent Bernoulli random variables such that
$$
\PP[X_k=1]=1-\PP[X_k=0]=\frac{1}{kf(k)^\alpha}.
$$
For any $a>1$, let $\mathcal{N}_n=|\EE_X\cap (a^n,a^{n+1}]|$,  then
\[
\limsup_{n\to\infty}\NN_n= \lfloor \alpha^{-1}\rfloor, \quad \text{ a.s},
\]
where $\lfloor \alpha^{-1}\rfloor$ is the largest integer not greater than $\alpha^{-1}$.
 \end{lemma}

 \begin{proof}
 For any $n\in\N$, set
\[
I_n:=\N \cap (a^n, a^{n+1}]
\text{ and }
L_n : = | I_n|.
\]
Observe that
\[
a^{n+1}-a^n-1\leq L_n\leq a^{n+1} - a^n+1,
\]
we have
\begin{align*}\label{L-n-a-n}
\lim\limits_{n\rightarrow\infty}\frac{L_n}{a^n}=a-1.
\end{align*}
For any $k\in I_n$, since $f$ is non-decreasing,  we have
\[
\PP[X_k=1]=\frac{1}{kf(k)^\alpha}\geq\frac{1}{a^{n+1}f(a^{n+1})^\alpha}
\]
and
\begin{align*}
& \PP[X_k=0]  \geq \PP\left[X_{\lfloor a^n\rfloor + 1}=0\right]
                            \\&=1-\frac{1}{\left(1+\lfloor a^n\rfloor\right) f (\lfloor a^n\rfloor + 1)^\alpha}
                            \geq 1-\frac{1}{a^n f(a^n )^\alpha}.
\end{align*}
Let $C=\lfloor \frac{1}{\alpha}\rfloor$, then $\alpha C\leq 1$. Moreover, we have
\begin{align*}
\PP[\NN_n=C]\geq &\sum_{ \substack{A\subset I_n \\ |A|=C}} \left[\frac{1}{a^{n+1} f( a^{n+1})^\alpha}\right]^C\left[1-\frac{1}{a^n f( a^n)^\alpha}\right]^{L_n-C}\\
=&\binom {L_n}{ C}\left[\frac{1}{a^{n+1} f( a^{n+1})^\alpha}\right]^C\left[1-\frac{1}{a^n f( a^n)^\alpha}\right]^{L_n-C}.
\end{align*}
Notice that
\[
\lim_{n\to\infty}\left[1-\frac{1}{a^n f(a^n)^\alpha}\right]^{L_n-C}=\lim_{n\to\infty}\left\{\left[1-\frac{1}{a^n f( a^n)^\alpha}\right]^{a^n f( a^n)^\alpha}\right\}^{\frac{L_n-C}{a^n f( a^n)^\alpha}}= 1.
\]
Therefore, there exist constants $\beta>0$ and $M_0>0$ such that for $n\geq M_0$,
\[
\PP[\NN_n=C]\geq \beta\binom {L_n}{ C}\left[\frac{1}{a^{n+1}f( a^{n+1})^\alpha}\right]^C.
\]
 Recall
\[
a^n(a-1)-1\leq L_n\leq a^n(a-1)+1,
\]
we have
\begin{align*}
\binom {L_n}{ C}\frac{1}{a^{(n+1)C}}=\frac{\Gamma(L_n+1)}{\Gamma(C+1)\Gamma(L_n-C+1)}\frac{1}{a^{(n+1)C}}
\geq \\
\geq \frac{1}{\Gamma(C+1)}\frac{\Gamma(a^n(a-1))}{\Gamma(a^n(a-1)+2-C)}\frac{1}{a^{(n+1)C}}.
\end{align*}
Notice that for any real number $x$,
\[
\lim\limits_{n\to\infty}\frac{\Gamma(n+x)}{\Gamma(n)n^x}=1.
\]
Then we can find constants $M_1>0$  and $\beta'>0$ such that
\[
\binom {L_n}{ C}\frac{1}{a^{(n+1)C}}>\beta' \qquad \text{ for all }  n \ge M_1.
\]
Since $\alpha C \le 1$, we obtain for $n\ge M:=\max\{M_0,M_1\}$,
\[
\PP[\NN_n=C]\geq \frac{\beta \beta'}{ f( a^{n+1})^{\alpha C}}\geq \frac{ \beta \beta'}{f(a^{n+1})}.
\]
Since $f$ is an admissible function, we have
\begin{align*}
\sum_{n=M}^{\infty}\frac{1}{f(a^{n+1})}
 &= \frac{1}{ \log a}\sum_{n=M}^{\infty} \frac{1}{f(a^{n+1})} \int_{a^{n+1}}^{a^{n+2}}\frac{dx}{x}\geq \\
 &\geq\frac{1}{\log a}\sum_{n=M}^{\infty}  \int_{a^{n+1}}^{a^{n+2}}\frac{dx}{xf(x)}
\geq\frac{1}{\log a}\int_{a^{M+1}}^{\infty}
\frac{dx}{x f(x)} = \infty.
\end{align*}
It follows that
\[
\sum_{n=1}^{\infty}\PP[\NN_n=C]=\infty.
\]
Observe that the random variables $\NN_n$ are independent, we have
\begin{align}\label{ge-C-as}
\limsup_{n\to\infty} \NN_n\geq C, \quad a.s.
\end{align}
On the other hand,
\[
\PP[\NN_n\geq C+1] = \sum_{k = C+1}^{L_n} \PP[\NN_n = k]  \leq \sum_{k = C+1}^{L_n} \binom{L_n}{k}   \left[ \frac{1}{a^{n}f(a^n)^{\alpha}} \right]^k.
\]
Then by using $C+1 > \frac{1}{\alpha}$ and $f(a^n) \ge 1$, we obtain
\begin{align*}
\sum_{n\in\N}\PP[\NN_n\geq C+1]& \leq \sum_{n\in\N}\sum_{k= C+1}^{L_n}\binom{L_n}{k}\frac{1}{a^{nk} f(a^n)^{\alpha k}}  \leq \sum_{n\in\N} \frac{1}{f(a^n)^{\alpha (C+1)}}\sum_{k= C+1}^{L_n}\binom{L_n}{k}\frac{1}{a^{nk}}
\\
&  \leq \sum_{n\in\N} \frac{1}{f(a^n)^{\alpha (C+1)}} \left[1+\frac{1}{a^n}\right]^{L_n}
\leq  \sup_{n\in \N} \left[1+\frac{1}{a^n}\right]^{a^{n+1}} \bigg( \sum_{n\in\N}\frac{1}{f(a^n)^{\alpha(C+1)}}\bigg).
\end{align*}
Since $\alpha(C+1)>1$ and $f$ is admissible, we have
\begin{align*}
 \log a \bigg(\sum_{n=2}^{\infty}\frac{1}{f(a^n)^{\alpha(C+1)}} \bigg)= \sum_{n=2}^{\infty}\frac{1}{f(a^n)^{\alpha(C+1)}} \int_{a^{n-1}}^{a^n} \frac{dx}{x}\leq \\
\leq\sum_{n=2}^{\infty}\int_{a^{n-1}}^{a^n}
\frac{dx}{xf(x)^{\alpha (C+1)}}=\int_a^{\infty}\frac{dx}{xf(x)^{\alpha (C+1)} }<\infty.
\end{align*}
It follows that
\[
\sum_{n\in\N}\PP[\NN_n\geq C+1]<\infty
\]
and hence
\begin{align}\label{le-C-as}
\limsup_{n\to\infty} \NN_n<C+1, \quad a.s.
\end{align}
Combining \eqref{ge-C-as} and \eqref{le-C-as} and using the fact  that $\limsup\limits_{n\to\infty}  \NN_n \in \N \cup \{\infty\}$, we obtain the desired limit equality
\[
\limsup_{n\to\infty} \NN_n =  C = \lfloor \alpha^{-1} \rfloor, \quad a.s.
\]
This completes the proof of Lemma \ref{P:Cmany}.
 \end{proof}

Before proving Theorem \ref{P:lacunary}, we present a corollary of Lemma \ref{P:Cmany}.
\begin{corollary}
Let $f$ be an admissible function.
Let $g:[1,\infty)\to [1,\infty)$  be a function such that $g(k)\leq f(k)^\alpha$ for any $0<\alpha\leq 1$. Consider a sequence of independent Bernoulli random variables $X=(X_k)_{k\in\mathbb{N}}$ with
$$
\PP[X_k=1]=1-\PP[X_k=0]=\frac{1}{kg(k)}.
$$
Then almost surely $\EE_X$ is not a finite union of lacunary sequence.
\end{corollary}
\begin{proof}
The proof follows from Lemma \ref{P:Cmany} and the standard monotone coupling arguments.
\end{proof}
Now, we are ready to prove Theorem \ref{P:lacunary}.

\begin{proof}[Proof of Theorem \ref{P:lacunary}]
First, we prove  equality \eqref{sup-prod}. It suffices to show
\begin{align}\label{prod-small-prob}
\sum_{n=1}^{\infty}\mathbb{P}(\NN_n\cdot \NN_{n+1}\geq 1)<\infty.
\end{align}
Since for each $n\in\mathbb{N}$, $\NN_n$ and $\NN_{n+1}$ are independent, we have
$$
\mathbb{P}(\NN_n\cdot \NN_{n+1}\geq 1)=\mathbb{P}(\NN_n\geq 1)\mathbb{P}(\NN_{n+1}\geq 1)=[1-\mathbb{P}(\NN_n=0)][1-\mathbb{P}(\NN_{n+1}=0)].
$$
Let $I_n= \N \cap (2^n, 2^{n+1}]$, then we have
$$
\mathbb{P}(\NN_n=0)=\prod_{k\in I_n}\Big(1-\frac{1}{kf(k)}\Big).
$$
Since $f(x)\geq 1$, there exists $\beta>0$ such that
$$
1-\frac{1}{kf(k)}\geq \exp\Big(-\frac{\beta}{kf(k)}\Big) \quad \text{for all integers $k\geq 1$.}
$$
 It follows that for $n\geq 1$
$$
\mathbb{P}(\NN_n=0)\geq \exp\Big(-\beta\sum_{k=2^n+1}^{2^{n+1}}\frac{1}{k f(k)}\Big)\geq \exp\Big(-\beta
\underbrace{\int_{2^n}^{2^{n+1}}\frac{1}{x f(x)}dx}_{ \text{denoted by $A_n$}}\Big).
$$
Observe that $1-e^{-x}\leq x$ for $x\geq 0$, then
$1-\exp(-\beta A_n)\leq \beta A_n$,
and thus
\begin{eqnarray*}
\mathbb{P}(\NN_n \NN_{n+1}\geq 1)\leq \beta^2 A_nA_{n+1}.
\end{eqnarray*}
Since
$$
A_n=\int_{2^n}^{2^{n+1}}\frac{1}{x f(x)}dx\leq\frac{2^{n+1}-2^n}{2^{n}f(2^n)} = \frac{1}{f(2^{n})},
$$
we have
$$
A_nA_{n+1}\leq \frac{1}{f(2^{n})f(2^{n+1})}\leq\frac{1}{f(2^{n})^2}.
$$
It follows that there exists $M_0\in \N$ such that
$$\sum_{n=M_0}^{\infty}A_nA_{n+1}\leq \sum_{n=M_0}^{\infty}\frac{1}{f(2^{n})^2}<\infty.
$$
The last inequality is due to the following inequality
\begin{align*}
\sum_{n=1}^{\infty}\frac{1}{f(2^{n})^2}  = \sum_{n  = 1}^\infty   \frac{2^{n} - 2^{n-1} }{ 2^{n-1} f(2^{n})^2} \le \sum_{n  = 1}^\infty \int_{2^{n-1}}^{2^{n}}  \frac{1}{x f( x)^2} dx = \int_1^{\infty}\frac{1}{xf(x)^2}dx<\infty.
\end{align*}
Consequently, we obtain the desired inequality \eqref{prod-small-prob} and complete the proof of \eqref{sup-prod}.

Next, we prove the assertion on lacunary properties. We write $\EE_X$ as an increasing sequence $\EE_X = \{n_k(X)\}$.  By Lemma \ref{P:Cmany},
\[
\limsup\limits_{n\to\infty} \NN_n=1 \text{ a.s. }
\]
Combining this with  \eqref{sup-prod},  we derive that almost surely,
\[
\limsup_{k\to\infty} \frac{n_{k+1}(X) }{n_k(X)} \geq 2.
\]
Therefore,  almost surely,  we have
\[
\inf_{k\in \N} \frac{n_{k+1}(X)}{n_k(X)} >1.
\]
That is, $\EE_X$ is a lacunary sequence. This completes the proof of Theorem \ref{P:lacunary}.
\end{proof}

\section{The gap}
The goal of this section is to prove Propositions \ref{P:gap} and \ref{P:boundedgap}.

\begin{proof}[Proof of Proposition \ref{P:gap}]
The assumption $\sum_{k}p_k=\infty$  implies that the sequence  $\EE_X$  is almost surely an infinite subset. 
Fix $C\in\N$, consider the event
\[
E:=\left\{\liminf_{k\rightarrow\infty}\left( n_{k+1}(X)-n_k (X)\right) \leq C\right\}.
\]
Notice that
\[
E=  \mathop{\bigcup}_{l=1}^C \left\{\liminf_{k\rightarrow\infty}\left( n_{k+1}(X)-n_k (X)\right)  = l \right\} = \mathop{\bigcup}_{l=1}^C\left\{ n_{k+1}(X)-n_k(X)=l\quad\text{i.o.}\right\}
\]
and, for any fixed integer $1\le l\le C$, 
\begin{align*}
\PP\left[ n_{k+1}(X)-n_k(X)=l\quad\text{i.o.}\right]&\leq\PP\left[X_{k}X_{k+l}=1\,\,\text{i.o.}\right].
\end{align*}
Since 
\begin{align*}
\sum_{k=1}^\infty\PP\left[X_{k}X_{k+l}=1\right] = \sum_{k=1}^\infty\E[X_{k}X_{k+l}]   =  \sum_{k=1}^\infty p_kp_{k+l} \le \big(\sum_{k = 1}^\infty p_k^2\big)^{1/2} \big(\sum_{k = 1}^\infty p_{k+l}^2\big)^{1/2}
<\infty,
\end{align*}
we have
\[
\PP\left[X_{k}X_{k+l}=1\,\,\text{i.o.}\right]=0.
\]
Hence $\PP[E]=0$ and we get almost surely
\[
\liminf_{k\rightarrow\infty}(n_{k+1}(X)-n_k(X))> C.
\]
Since $C$ is an arbitrary number,
\[
\lim_{k\rightarrow\infty}(n_{k+1}(X)-n_k(X))=\infty \quad \text{a.s.}
\]
This completes the proof of Proposition \ref{P:gap}.
\end{proof}

\begin{proof}[Proof of Proposition \ref{P:boundedgap}]
Let $S = \{s_1<s_2<\cdots\}\subset\N$ be a given  sequence with bounded gap such that $\mathrm{Gap}(S)=C<\infty$.
Since $p_k$ is decreasing, we have
\begin{align*}
\sum_{i\in \N\setminus S}p_i=&\sum_{k=1}^\infty \sum_{s_{k}<i < s_{k+1} }p_i\leq C\sum_{k=1}^\infty p_{s_{k}}=C\sum_{i\in S}p_i.
\end{align*}
the above inequality, combined with $\sum_{i\in S}p_i+\sum_{i\in \N\setminus S}p_i = \sum_{i\in \N}p_i = \infty$, implies that 
\[
\sum_{i\in S}p_i=\infty.
\]
Therefore, 
\[
\sum_{i\in S}\log\left( \frac{1}{1-p_i}\right)=\infty.
\]
By the independence of $X$,
\begin{align*}
\PP[S\cap\EE_X=\emptyset]=\prod_{i\in S}(1-p_i)=\exp\Big\{-\sum_{i\in S}\log\Big( \frac{1}{1-p_i}\Big)\Big\}=0.
\end{align*}
Then we complete the proof of Proposition \ref{P:boundedgap}.
\end{proof}

\section{Arithmetic progression}
We prove Propositions \ref{P:AP01} and \ref{P:AP02} in this section.
\begin{proof}[Proof of Proposition \ref{P:AP01}]
Fix any  $l\geq 3$ and  for any $k\in\N$, define
\[
E(k):=\{i+kl\in \EE_X: i=1,2,\dots,l\}.
\]
Then the events $\{E(k):k=0,1,2,\dots\}$ are independent. Moreover,
\begin{align*}
&\sum_{k=0}^\infty \PP[E(k)]
=\sum_{k=0}^\infty\prod_{i=1}^l\frac{1}{e^{\left(\log(i+kl)\right)^\epsilon}}
=\sum_{k=0}^\infty\frac{1}{e^{\sum_{i=1}^l\left(\log(i+kl)\right)^\epsilon}}\geq \\
\geq&\sum_{k=0}^\infty\frac{1}{e^{l\left(\log(l+kl)\right)^\epsilon}}
\geq c\sum_{k=0}^\infty\frac{1}{e^{\log(k+1)}}
=\sum_{k=0}^\infty\frac{1}{k+1}=\infty.
\end{align*}
It follows from  Borel-Cantelli lemma \cite[Theorem 2.3.6]{Durrett} that
\[
\PP[E(k)\, \mathrm{i.o.}]=1.
\]
This means that $\EE_X$ contains infinitely many arithmetic progressions of length $l$ and completes the proof.
\end{proof}

\begin{proof}[Proof of Proposition \ref{P:AP02}]
Let $0 < \alpha \le 1/2$ and let $l \ge 2$ be the smallest integer with $l >  \alpha^{-1} \ge 2$.
For  any $i, d\in\N$,  consider the event
\[
E_l(i, d)=\{i+kd\in\EE_X: 0\leq k\leq l\}.
\]
Since $p_k=O(k^{-\alpha})$, there exists a constant $C>0$ such that
\[
\PP[E_l(i,d)]=\prod_{k=0}^lp_{i+kd}\leq C \prod_{k=0}^l \frac{1}{(i+kd)^\alpha}      = C  \frac{1}{i^\alpha} \frac{1}{(i+d)^{l\alpha}}.
\]
Then we have
\begin{align*}
\sum_{i, d=1}^\infty\PP[E_l(i, d)] &  \le   C  \sum_{i=1}^\infty \frac{1}{i^\alpha}  \sum_{d=1}^\infty \frac{1}{(i + d)^{l\alpha}}   \le C \sum_{i = 1}^\infty \frac{1}{i^\alpha} \int_{i}^\infty \frac{dx}{x^{l\alpha}}  = \frac{C}{l\alpha -1}  \sum_{i=1}^\infty  \frac{1}{i^{l \alpha }}<\infty.
\end{align*}
It follows from Borel-Cantelli lemma that, almost surely, there exist only finitely many arithmetic progressions of length $l+1$ contained in $\EE_X$. This completes the proof.
\end{proof}

%\bibliography{RRS}

\end{document}